\documentclass[10pt,a4paper]{article}
\usepackage[cp1251]{inputenc}

\usepackage{amssymb}
\usepackage{amsmath}
\usepackage{amsthm}
\usepackage{rotate}
\usepackage{graphicx}
\usepackage{psfrag}
\usepackage[T1]{fontenc}
\usepackage{fancyhdr}
\usepackage{afterpage}

 
\newcommand{\mip}[1]{\begin{minipage}[c]{7em}#1\end{minipage}}
\newcommand{\med}[1]{\begin{minipage}[c]{15em}#1\end{minipage}}
\newcommand{\mad}[1]{\begin{minipage}[c]{5em}#1\end{minipage}}
\newcommand{\myd}[1]{\begin{minipage}[c]{11em}#1\end{minipage}}
\newcommand{\mep}[1]{\begin{minipage}[c]{10em}#1\end{minipage}}

\begin{document}
\sloppy
\newtheorem{theorem}{Theorem}
\newtheorem{acknowledgement}{Acknowledgement}
\newtheorem{algorithm}{Algorithm}
\newtheorem{axiom}{Axiom}
\newtheorem{case}{Case}
\newtheorem{claim}{Claim}
\newtheorem{conclusion}{Conclusion}
\newtheorem{condition}{Condition}
\newtheorem{conjecture}{Conjecture}
\newtheorem{corollary}[theorem]{Corollary}
\newtheorem{criterion}{Criterion}
\newtheorem{example}{Example}
\newtheorem{exercise}[theorem]{Exercise}
\newtheorem{lemma}[theorem]{Lemma}
\newtheorem{notation}{Proposition}
\newtheorem{problem}{Problem}
\newtheorem{proposition}[theorem]{Proposition}
\newtheorem{remark}{Remark}
\newtheorem{solution}{Solution}
\newtheorem{summary}{Summary}
\theoremstyle{definition}
\newtheorem{definition}{Definition}
\newcommand{\un}[1]{{\rm{un}}{#1}}
\newcommand{\tb}[2]{\begin{tabular}{lr}#1&\\&#2\end{tabular}}
\newcommand{\tbb}[2]{\begin{tabular}{lr}&#1\\#2&\end{tabular}}
\renewcommand{\labelenumi}{\theenumi)}
\newcommand{\pn}[1]{\begin{minipage}[t]{0.7in}\renewcommand\baselinestretch{0.8}\small #1\end{minipage}}
\newcommand{\pnn}[1]{\begin{minipage}[t]{0.95in}\renewcommand\baselinestretch{0.8}\small #1\end{minipage}}

\author{Halyna Krainichuk }
\title{Classification of group isotopes according to their symmetry groups}
\date{}
\maketitle

\abstract{\footnotesize The class of all quasigroups is covered by six classes: the class of all asymmetric quasigroups and five varieties of quasigroups (commutative, left symmetric, right symmetric, semi-symmetric and totally symmetric). Each of these classes is characterized by symmetry groups of its quasigroups.

In this article, criteria of belonging of group isotopes to each of these classes are found, including the corollaries for linear, medial and central quasigroups etc. It is established that an isotope of a noncommutative group is either semi-symmetric or asymmetric, each non-medial T-quasigroup is asymmetric etc. The obtained results are applied for the classification of linear group isotopes of prime orders, taking into account their up to isomorphism description.

}

\emph{Classification:} 20N05, 05B15.

\section*{\centerline{Introduction}}
\indent

The group isotope variety is an abundant class of quasigroups in a sense that it contains quasigroups from almost all quasigroup classes, which have ever been under consideration. Many authors obtained their results in the group isotope theory that is why the available results are widely scattered in many articles and quite often are repeated. However, taking into account the group theory development, these results have largely a ``folkloric'' level of complexity. Considering their applicability, they should be systematized. One of the attempts of systemic presentations of the group isotopes is the work by Fedir Sokhatsky ``On group isotopes'', which is given in three articles~\cite{sokha19951,sokha19952,sokha19963}. But, parastrophic symmetry has been left unattended. 

Some a concept of a symmetry for all parastrophes of quasigroup was investigated by J.D.H.~Smith~\cite{smith2007}. This symmetry is known as \emph{triality}. The same approach for all parastrophes of quasigroup can be found in such articles as T.~Evans~\cite{evans1951}, V.~Belousov~\cite{bil1967}, Yu.~Movsisyan~\cite{movsis1998}, V.~Shcherbacov~\cite{shcherb2011}, G.~Belyavskaya and T.~Popovych~\cite{belyav_popovych2012} and in many articles of other authors. Somewhat different idea of a symmetry of quasigroups and loops was suggested by F.~Sokhatsky~\cite{sokha2013}. Using this concept, one can systematize many results.

The purpose of this article is a classification of group isotopes according to their parastrophic symmetry groups.

If a $\sigma$-parastrophe coincides with a quasigroup itself, then $\sigma$ is called a \emph{symmetry} of the quasigroup. The set of all symmetries forms a group, which is a subgroup of the symmetry group of order three, i.e., $S_{3}$.  All quasigroups defined on the same set are distributed into 6 blocks according to their symmetry groups. The class of quasigroups whose symmetry contains the given subgroup of the group $S_{3}$ forms a variety. Thus, there exist five of such varieties: commutative, left symmetric, right symmetric, semi-symmetric and totally symmetric. The rest of quasigroups forms the class of all asymmetric quasigroups, which consists of quasigroups with a unitary symmetry group.

In this article, the following results are obtained:
\begin{itemize}
  \item criteria of belonging of group isotopes to each of these classes (Theorem~\ref{tcgi});
  \item a classification of group isotopes defined on the same set, according to their symmetry groups (Table 1 and Corollary~\ref{cisp});
  \item corollaries for the well-known classes, such as linear (Corollary~\ref{cisp}), medial (Corollary~\ref{cismed}) and central quasigroups (Corollary~\ref{ctqis});
  \item every non-medial T-quasigroup is asymmetric (Corollary~\ref{cisnotmed});
 \item an isotope of a noncommutative group is either semi-symmetric or asymmetric (Corollary~\ref{cisnotcom});
 \item classification of linear isotopes of finite cyclic groups (Corollary~\ref{cisp});
 \item classification of linear group isotopes of prime orders (Theorem~\ref{tispm}).
\end{itemize}

\section{Preliminaries}
\indent

A groupoid $(Q;\cdot)$ is called a quasigroup, if for all $a$, $b\in Q$ every of the equations $x\cdot a=b$ and $a\cdot y=b$ has a unique solution.
For every $\sigma\in S_{3}$ a $\sigma$-parastrophe $(\stackrel{\sigma}{\cdot})$ is defined by
$$
x_{1\sigma}\stackrel{\sigma}{\cdot}x_{2\sigma}=x_{3\sigma}\;\Longleftrightarrow\;x_{1}\cdot x_{2}=x_{3},
$$
where $S_{3}:=\{\iota,s,\ell,r,s\ell,sr\}$ is the symmetric group of order $3$ and $s:=(12)$, $\ell:=(13)$, $r:=(23)$.

A mapping $(\sigma;(\cdot))\;\mapsto\;(\,\stackrel{\sigma}{\cdot})$ is an action on the set $\Delta$ of all quasigroup operations defined on $Q$. A stabilizer $\mathrm{Sym}(\cdot)$ is called a \emph{symmetry group} of $(\cdot)$. Thus, the number of different parastrophes of a quasigroup operation $(\cdot)$ depends on its symmetry group $\mathrm{Sym}(\cdot)$. Since $\mathrm{Sym}(\cdot)$ is a subgroup of the symmetric group $S_{3}$, then there are six classes of quasigroups. A quasigroup is called\\[-2ex]
\begin{itemize}
  \item \emph{asymmetric}, if $\mathrm{Sym}(\cdot)=\{\iota\}$, i.e., all parastrophes are pairwise different;
  \item \emph{commu\-ta\-ti\-ve}, if $\mathrm{Sym}(\cdot)\supseteq\{\iota,s\}$, i.e., the class of all commutative quasigroups is described by $xy=yx$ it means that\\
      $$(\cdot)=(\stackrel{s}{\cdot}),\quad (\stackrel{\ell}{\cdot})=(\stackrel{sr}{\cdot}),\quad (\stackrel{r}{\cdot})=(\stackrel{s\ell}{\cdot});$$
  \item \emph{left symmetric}, if $\mathrm{Sym}(\cdot)\supseteq\{\iota,r\}$, i.e., the class of all left symmetric quasigroups  is described by $x\cdot xy=y$, it means that\\
      $$(\cdot)=(\stackrel{r}{\cdot}),\quad (\stackrel{s}{\cdot})=(\stackrel{\ell}{\cdot}), \quad (\stackrel{s\ell}{\cdot})=(\stackrel{sr}{\cdot});$$
  \item \emph{right symmetric}, if $\mathrm{Sym}(\cdot)\supseteq\{\iota,\ell\}$, i.e., the class of all right symmetric quasigroups  is described by $xy\cdot y=x$, it means that\\
      $$(\cdot)=(\stackrel{\ell}{\cdot}),\quad (\stackrel{s}{\cdot})=(\stackrel{r}{\cdot}),\quad (\stackrel{sr}{\cdot})=(\stackrel{s\ell}{\cdot});$$
  \item \emph{se\-mi-sym\-met\-ric}, if $\mathrm{Sym}(\cdot)\supseteq A_{3}$, i.e., the class of all semi-symmetric quasigroups is described by $x\cdot yx=y$, it means that\\
       $$(\cdot)=(\stackrel{s\ell}{\cdot})=(\stackrel{sr}{\cdot}),\quad (\stackrel{s}{\cdot})=(\stackrel{\ell}{\cdot})=(\stackrel{r}{\cdot});$$
  \item \emph{totally sym\-met\-ric}, if $\mathrm{Sym}(\cdot)=S_{3}$, i.e.,  the class of all totally symmetric quasigroups  is described by $xy=yx$ and $xy\cdot y=x$, it means that all parastrophes coincide.
\end{itemize}

We will say, a\emph{ quasigroup has the property of a symmetry}, if it satisfies one of the following symmetry properties: commutativity, left symmetry, right symmetry, semi-symmetry or total symmetry.

Let $P$ be an arbitrary proposition in a class of quasigroups ${\mathfrak{A}}$. The proposition ${^{\sigma\!}P}$ is said to be a $\sigma$-parastrophe of $P$, if it can be obtained from $P$ by replacing every parastrophe $(\stackrel{\tau}{\cdot})$ with $(\stackrel{\tau\sigma^{-1}}{\!\!\cdot})$; ${^{\sigma\!}\mathfrak{A}}$ denotes the class of all $\sigma$-parastrophes of quasigroups from ${\mathfrak{A}}$.

\begin{theorem}\label{tpp}$\cite{sokha2013}$
Let $\mathfrak{A}$ be a class of quasigroups, then a proposition $P$ is true in ${\mathfrak{A}}$ if and only if ${^{\sigma\!}\!P}$ is true in ${^{\sigma\!}\mathfrak{A}}$.
\end{theorem}

\begin{corollary}\label{c1pp}$\cite{sokha2013}$ Let $P$ be true in a class of quasigroups ${\mathfrak{A}}$, then ${^{\sigma\!}P}$ is true in ${^{\sigma\!}\mathfrak{A}}$ for all $\sigma\in\mathrm{Sym}({\mathfrak{A}})$.
\end{corollary}

\begin{corollary}\label{c2pp}$\cite{sokha2013}$ Let $P$ be true in a totally symmetric class ${\mathfrak{A}}$, then ${^{\sigma\!}P}$ is true in ${\mathfrak{A}}$ for all $\sigma$.
\end{corollary}

A groupoid $(Q;\cdot)$ is called an \emph{isotope of a groupoid} $(Q;+)$ iff there exists a triple of bijections $(\alpha,\beta,\gamma)$, which is called an \emph{isotopism}, such that the relation $x\cdot y:=\gamma^{-1}(\alpha{x}+\beta{y})$ holds. An isotope of a group is called a \emph{group isotope}.

A permutation $\alpha$ of a set $Q$ is called {\it unitary} of a group $(Q;+)$, if $\alpha(0)=0$, where $0$ is a neutral element of $(Q;+)$.

\begin{definition}\label{dcdec} \cite{sokha19952}
Let $(Q;\cdot)$ be a group isotope and $0$ be an arbitrary element of $Q$, then the right part of the formula
\begin{equation}\label{cdec}
  x\cdot y=\alpha x+a+\beta y
\end{equation}
is called a {\it $0$-canonical decomposition}, if  $(Q;+)$ is a group, $0$ is its neutral element and $\alpha$, $\beta$ are unitary permutations of $(Q;+)$.
\end{definition}

In this case, we say: the element $0$ defines the {\it canonical decomposition}; $(Q;+)$ is its {\it decomposition group}; $\alpha$, $\beta$ are its {\it coefficients} and $a$ is its {\it free member}.

\begin{theorem}\label{tcdec}$\cite{sokha19952}$
An arbitrary element of a group isotope uniquely defines a canonical decomposition of the isotope.
\end{theorem}

\begin{corollary}\label{cigc} $\cite{sokha19951}$
If a group isotope $(Q;\cdot)$ satisfies an identity
$$
w_{1}(x)\cdot w_{2}(y)=w_{3}(y)\cdot w_{4}(x)
$$
and the variables $x$, $y$ are quadratic, then $(Q;\cdot)$ is isotopic to a commutative group.
\end{corollary}

Recall, that a variable is \emph{quadratic} in an identity, if it has exactly two appearances in this identity. An identity is called \emph{quadratic}, if all variables are quadratic. If a quasigroup  $(Q;\cdot)$ is isotopic to a parastrophe of a quasigroup $(Q;\circ)$, then $(Q;\cdot)$ and $(Q;\circ)$ are called \emph{isostrophic}.

The given below Theorem~\ref{tauto} and its Corollary~\ref{cauto} are well known and can be found in many articles, for example, in~\cite{bil1967}, \cite{sokha19952}.

\begin{theorem}\label{tauto}
A triple $(\alpha,\beta,\gamma)$ of permutations of a set $Q$ is an autotopism of a group $(Q,+)$ if and only if there exists an automorphism $\theta$ of $(Q,+)$ and elements $b,c\in Q$ such that
$$
\alpha=L_{c}R_{b^{-1}}\theta, \quad \beta=L_{b}\theta, \quad \gamma=L_{c}\theta.
$$
\end{theorem}

\begin{corollary}\label{cauto}
 Let $\alpha$, $\beta_{1}$, $\beta_{2}$, $\beta_{3}$, $\beta_{4}$ be permutations of a set $Q$ besides $\alpha$ is a unitary transformation of a group $(Q,+)$ and let
$$
\alpha(\beta_{1}x+\beta_{2}y)=\beta_{3}u+\beta_{4}v,
$$
where $\{x,y\}=\{u,v\}$ holds for all $x,y\in Q$. Then the following statements are true:
\begin{enumerate}
  \item $\alpha$ is an automorphism of $(Q,+)$, if $u=x,\;v=y$;
  \item $\alpha$ is an anti-automorphism of $(Q,+)$, if $u=y,\;v=x$.
\end{enumerate}
\end{corollary}

A quasigroup is \emph{a linear group isotope}~\cite{bil1966}, if there exists a group $(Q;+)$, its automorphisms $\varphi$, $\psi$, an arbitrary element $c$ such that for all $x,y\in Q$
$$
x\cdot y=\varphi x+c+\psi y.
$$

T.~Kepka, P.~Nemec~\cite{kepnem1971I,kepnem1971II} introduced the concept of $T$-quasigroups and studied their properties, namely the class of all $T$-quasigroups is a variety. $T$-quasigroups, sometimes, are called \emph{central} quasigroups. Central quasigroups are precisely
the abelian quasigroups in the sense of universal algebra \cite{szen1998}. A \emph{$T$-quasigroup} is a linear isotope of an abelian group and, according to Toyoda-Bruck theorem~\cite{toyoda1941,bruck1944}, it is \emph{medial} if and only if coefficients of its canonical decompositions commute. A \emph{medial quasigroup}~\cite{bil1967} is a quasigroup defined by the identity of mediality
$$
xy\cdot uv=xu\cdot yv.
$$

\section{Isotopes of groups}
\indent

In this section, criteria for group isotopes to have symmetry properties are given and a classification of group isotopes according to their symmetry groups is described.

\subsection{Classification of group isotopes}

The criteria for group isotopes to be commutative, left symmetric and  right symmetric are found by O.~Kirnasovsky~\cite{kirnasov1995}. The criteria of total symmetry, semi-symmetry and asymmetry are announced by the author in~\cite{krain2014}.

The following theorem systematizes all criteria on symmetry and implies a classification of group isotopes according to their symmetry groups.

\begin{theorem}\label{tcgi}
Let $(Q;\cdot)$ be a group isotope and (\ref{cdec}) be its canonical decomposition, then
\begin{enumerate}
 \item \label{cis} $(Q;\cdot)$ is commutative if and only if $(Q;+)$ is abelian and $\beta=\alpha$;
 \item \label{lis} $(Q;\cdot)$ is left symmetric if and only if $(Q;+)$ is abelian and $\beta=-\iota$;
\item \label{ris} $(Q;\cdot)$ is right symmetric  if and only if $(Q;+)$ is abelian and $\alpha=-\iota$;
\item \label{tis} $(Q;\cdot)$ is totally symmetric  if and only if $(Q;+)$ is abelian and $\alpha=\beta=-\iota$;
\item \label{sis} $(Q;\cdot)$ is semi-symmetric if and only if  $\alpha$ is an anti-automorphism of $(Q;+)$,\\
  $\beta=\alpha^{-1}$, $\alpha^{3}=-I_{a}^{-1}$, $\alpha{a}=-a$, where $I_{a}(x):=-a+x+a$;
\item \label{ais} $(Q;\cdot)$ is asymmetric if and only if $(Q;+)$ is not abelian or $-\iota\neq \alpha\neq \beta\neq -\iota$ and at least one of the following conditions is true: $\alpha$ is not an anti-automorphism, $\beta\neq\alpha^{-1}$, $\alpha^{3}\neq -I_{a}^{-1}$, $\alpha{a}\neq -a$.
\end{enumerate}
\end{theorem}

\begin{proof} {\it\ref{cis})} Let a group isotope $(Q;\cdot)$ be commutative, e.i., the identity $xy=yx$ holds. Using its canonical decomposition (\ref{cdec}), we have
$$
\alpha{x}+a+\beta{y}=\alpha{y}+a+\beta{x}.
$$
Corollary~\ref{cigc} implies that $(Q;+)$ is an abelian  group. When $x=0$, we obtain $\alpha=\beta$.

Conversely, let $(Q;+)$ be an abelian group and $\beta=\alpha$, then
 $$
 x\cdot y=\alpha{x}+a+\alpha{y}=\alpha{y}+a+\alpha{x}=y\cdot x.
 $$
Thus, $(Q;\cdot)$ is a commutative quasigroup.

 {\it\ref{lis})} Let a group isotope $(Q;\cdot)$ be left symmetric, e.i., the identity $xy\cdot y=x$ holds. Using its canonical decomposition (\ref{cdec}), we have
 $$
 \alpha(\alpha{x}+a+\beta y)+a+\beta y=x.
 $$
Replacing $a+\beta y$ with $y$, we obtain $\alpha(\alpha{x}+y)=x-y$. Corollary~\ref{cauto} implies that $\alpha$ is an automorphism. When $x=0$, we have $\alpha y=-y$, i.e.,  $\alpha=-\iota$. Since $\alpha$ is an automorphism and an anti-automorphism, then $(Q;+)$ is an abelian group.

Conversely, suppose that the conditions of  {\it\ref{lis})} are performed, then $(Q;\cdot)$ is a left symmetric quasigroup. Indeed,
  $$
  xy\cdot y=-(-x+\beta y)+\beta y=x-\beta y+\beta y=x.
  $$

The proof of {\it\ref{ris})} is similar to {\it\ref{lis})}. The point {\it\ref{tis})} follows from {\it\ref{lis})} and {\it\ref{ris})}.

{\it\ref{sis})} Let a group isotope $(Q;\cdot)$ be semi-symmetric, e.i., the identity $x\cdot yx=y$ holds. Using (\ref{cdec}), we have
$$
\alpha x+a+\beta(\alpha y+a+\beta x)=y,
$$
hence,
$$
\beta(\alpha y+a+\beta x)=-a-\alpha x+y.
$$
Corollary~\ref{cauto} implies that  $\beta$ is an anti-automorphism of $(Q;+)$, therefore
\begin{equation}\label{dd1}
\beta^{2}x+\beta a+\beta\alpha y=-a-\alpha x+y.
\end{equation}
When $x=y=0$, we obtain $\beta a=-a$ and, when $x=0$, we have $\beta\alpha=\iota$, i.e.,
 $\beta=\alpha^{-1}$. Substitute the obtained relations in (\ref{dd1}):
$$
\alpha^{-2}x-a+y=-a-\alpha x+y.
$$
Reducing $y$ on the right in the equality and replacing $x$ with $\alpha^{2}x$, we have $x-a=-a-\alpha^{3}x$,
wherefrom $-\alpha^{3}x=a+x-a$ that is $\alpha^{3}=-I_{a}^{-1}$.

Conversely, suppose that the conditions of {\it\ref{sis})} hold. A quasigroup $(Q;\cdot)$ defined by $x\cdot y:=\alpha x+a+\alpha^{-1}y$
is semi-symmetric. Indeed,
$$
x\cdot yx=\alpha x+a+\alpha^{-1}(\alpha y+a+\alpha^{-1}x).
$$
Since $\alpha$ is an anti-automorphism of the $(Q;\cdot)$, then
$$
x\cdot yx=\alpha x+a+\alpha^{-2}x+\alpha^{-1}a+y.
$$
$\alpha^{3}=-I_{a}^{-1}$ implies $a+\alpha^{-2}x=-\alpha x+a$, so,
$$
x\cdot yx=\alpha x-\alpha x+a+\alpha^{-1}a+y=a+\alpha^{-1}a+y.
$$
Because $\alpha^{-1}a=-a$, then $x\cdot yx=y$. Thus, $(Q;\cdot)$ is semi-symmetric.

{\it\ref{ais})} Asymmetricity of a group isotope means that it is neither commutative, nor left symmetric, nor right symmetric, nor semi-symmetric. It means that all conditions {\it\ref{cis})}--{\it\ref{tis})} are false. Falsity of {\it\ref{cis})} implies falsity of {\it\ref{tis})}. Falsity of {\it\ref{sis}) } is equivalent to fulfillment of at least one of the conditions: $\alpha$ is not an anti-automorphism, $\beta\neq \alpha^{-1}$, $\alpha^{3}\neq -I_{a}^{-1}$, $\alpha(a)\neq -a$. Falsity of {\it\ref{cis})}, {\it\ref{lis})}, {\it\ref{ris})} means that $(Q;+)$ is noncommutative or each of the following inequalities $\beta\neq\alpha$, $\beta\neq -\iota$, $\alpha\neq -\iota$ is true. Thus, {\it\ref{ais})} has been proved.
\end{proof}

From Theorem~\ref{tcgi}, we can deduce the corollary for the classification of group isotopes over a noncommutative group.

\begin{corollary}\label{cisnotcom}
An isotope of a noncommutative group is either semi-symmetric or asymmetric.
\end{corollary}

\begin{proof} Theorem~\ref{tcgi} implies that an isotope of a noncommutative group can be semi-symmetric or asymmetric. A group isotope can not be asymmetric and semi-symmetric simultaneously, because according to definition, a symmetry group of a semi-symmetric quasigroup is $A_{3}$ or $S_{3}$ and a symmetry group of an asymmetric quasigroup is $\{\iota\}$.
\end{proof}

\begin{corollary}\label{cismed}
Commutative, left symmetric, right symmetric and totally sym\-met\-ric linear isotopes of a group are medial quasigroups.
\end{corollary}

\begin{proof} The corollary follows from Theorem~\ref{tcgi}, because in every of these cases the decomposition group of a canonical decomposition of a group isotope is commutative and both of its coefficients are automorphisms according to assumptions and they commute.
\end{proof}

\begin{corollary}\label{cisnotmed}
A nonmedial linear isotope of an arbitrary group is either semi-symmetric or asymmetric.
\end{corollary}

\begin{proof} The corollary immediately follows from Theorem~\ref{tcgi} and Corollary~\ref{cismed}.
\end{proof}

\subsection{Classification of isotopes of abelian groups}

An isotope of a nonabelian group is either semi-symmetric or asymmetric (see Corollary~\ref{cisnotcom}). In other words, commutative, left symmetric, right symmetric and totally symmetric group isotopes exist only among isotopes of a commutative group. Consequently, it is advisable to formulate a corollary about classification of isotopes of commutative groups.

\begin{corollary}\label{ciscom}
Let $(Q;\cdot)$ be an isotope of a commutative group and (\ref{cdec}) be its canonical decomposition, then the following conditions are true: {\it\ref{cis})}-{\it\ref{tis})} of Theorem~\ref{tcgi} and
\begin{itemize}
             \item[$5')$] $(Q;\cdot)$ is semi-symmetric if and only if  $\alpha$ is an automorphism of $(Q;+)$,              $\beta=\alpha^{-1}$, $\alpha^{3}=-\iota$, $\alpha{a}=-a$;
             \item[$6')$] $(Q;\cdot)$ is asymmetric if and only if $-\iota\neq\alpha\neq \beta\neq -\iota$ and at least one of the following conditions is true: $\alpha$ is not an automorphism, $\beta\neq\alpha^{-1}$, $\alpha^{3}\neq -\iota$, $\alpha{a}\neq -a$.
           \end{itemize}
\end{corollary}

\begin{proof} The proof follows from Theorem~\ref{tcgi} taking into account that $(Q;+)$ is commutative.
\end{proof}

The varieties of medial and $T$-quasigroups are very important and investigated subclasses of the variety of all group isotopes. These quasigroups have different names in the scientific literature. For example, medial quasigroups also are called entropic or bisymmetry, and $T$-quasigroups are called central quasigroups~\cite{szen1998}, linear isotopes of commutative groups and etc.

The next statement gives a classification of varieties of $T$-quasigroups according to their symmetry groups.
\begin{corollary}\label{ctqis}
Let $(Q;\cdot)$ be a linear isotope of a commutative group and (\ref{cdec}) be its canonical decomposition, then the following conditions are true: {\it\ref{cis})}-{\it\ref{tis})} of Theorem~\ref{tcgi}, {$5'$)} of Corollary~\ref{ciscom} and
 \begin{itemize}
              \item[$6'')$] $(Q;\cdot)$ is asymmetric if and only if $-\iota\neq\alpha\neq\beta\neq -\iota$ and at least one of the following conditions is true: $\beta\neq\alpha^{-1}$, $\alpha^{3}\neq -\iota$, $\alpha{a}\neq -a$.
           \end{itemize}
\end{corollary}

\begin{proof} This theorem immediately follows from Corollary~\ref{ciscom}.
\end{proof}

\begin{corollary}\label{cisnotmed}
Every nonmedial T-quasigroup is asymmetric.
\end{corollary}

In other words, if a T-quasigroup is not asymmetric, then it is medial.

\section{Linear isotopes of finite cyclic groups}

A full description of all $n$-ary linear isotopes of cyclic groups up to an isomorphism is given by F.~Sokhatsky and  P.~Syvakivskyj~\cite{sokhasyvak1994}. All pairwise non-isomorphic group isotopes up to order 15 and a criterion of their existence are established by O.~Kirnasovsky~\cite{kirnasov1995}. Some algebraic properties of non-isomorphic quasigroups are studied by L.~Chiriac, N.~Bobeica and D.~Pavel~\cite{chbp2014} using the computer.

Let $Q$ be a set and $\mathrm{Is}(+)$ be a set of all isotopes of a group $(Q;+)$. Theorem~\ref{tcgi} does not give a partition of $\mathrm{Is}(+)$. But it is easy to see that only totally symmetric quasigroups are common for two classes of symmetric quasigroups, i.e., of group isotopes being not asymmetric.

To emphasize that a group isotope is not totally symmetric we add the word `\textit{strictly}'. For example, the term `\textit{a strictly commutative group isotopes}' means that it is commutative, but not totally symmetric.

Theorem~\ref{tcgi} implies that the conditions of exclusion of totally symmetric quasigroups from a set of group isotopes are the following: coefficients of its canonical decomposition do not equal $-\iota$ simultaneously. Generally speaking, a set $\mathrm{Is}(+)$ is parted into six subsets.

\begin{center}
\begin{tabular}{|c|c|c|c|}\hline
\mip{A group isotope $(Q;\cdot)$} &  \mad{its symmetry group\smallskip}  &\med{ conditions of its\\ canonical decomposition (\ref{cdec}) }   \\ \hline
\mip{\smallskip is strictly \\commutative\smallskip}& $\{\iota,s\}$ & \med{$(Q;+)$ is abelian, $\beta=\alpha\neq -\iota$}\\ \hline
\mip{\smallskip is strictly left \\symmetric\smallskip} & $\{\iota,r\}$ & \med{$(Q;+)$ is abelian, $\beta=-\iota\neq\alpha$}\\ \hline
\mip{\smallskip is strictly right symmetric\smallskip}&  $\{\iota,\ell\}$ & \med{$(Q;+)$ is abelian, $\alpha=-\iota\neq\beta$}  \\ \hline
\mip{is strictly \\semi-symmetric} & $A_{3}$ & \med{$\alpha$ is an anti-automorphism of $(Q;+)$, $\beta=\alpha^{-1}$, $\alpha{a}=-a$, $\alpha^{3}=-I_{a}^{-1}$, where $I_{a}(x):=-a+x+a$, $(Q;+)$ is non-abelian or $\alpha\neq -\iota$} \\ \hline
\mip{is totally \\symmetric} & $S_{3}$  & \med{$(Q;+)$ is abelian $\alpha=\beta=-\iota$} \\ \hline
\mip{is asymmetric}&  $\{\iota\}$ & \med{$\alpha$ is not an anti-automorphism of $(Q;+)$, $\beta\neq\alpha^{-1}$, $\alpha^{3}\neq -I_{a}^{-1}$, $\alpha{a}\neq -a$ and $(Q;+)$ is non-abelian or $-\iota\neq \alpha\neq \beta\neq -\iota$.}  \\ \hline
\end{tabular}
\end{center}
\begin{center}
Table 1. A partition of group isotopes.
\end{center}

Consider the set of all linear isotopes of a finite cyclic group. Their up to isomorphism description has been found by F.~Sokhatsky and P.~Syvakivsky~\cite{sokhasyvak1994}.

We will use the following notation: $\mathbb{Z}_{m}$ denotes the ring of integers modulo $m$; $\mathbb{Z}_{m}^{*}$ the group of invertible elements of the ring $\mathbb{Z}_{m}$; and $(\alpha,\beta,d)$, where $\alpha,\beta\in\mathbb{Z}_{m}^{*}$ and $d\in\mathbb{Z}_{m}$, denotes an operation $(\circ)$ which is defined on $\mathbb{Z}_{m}$ by the equality
\begin{equation}\label{liniso}
x\circ y=\alpha\cdot x+\beta\cdot y+d.
\end{equation}
Since every automorphism $\theta$ of the cyclic group $(\mathbb{Z}_{m};+)$ can be defined by $\theta(x)=k\cdot x$ for some $k\in\mathbb{Z}_{m}^{*}$, then linear isotopes of $(\mathbb{Z}_{m};+)$ are exactly the operations being defined by $(\ref{liniso})$, i.e., they are the triples $(\alpha,\beta,d)$.

\begin{theorem}$\cite{sokhasyvak1994}$\label{tliniso2}
An arbitrary linear isotope of a cyclic $m$ order group is isomorphic to exactly one isotope $(\mathbb{Z}_{m},\circ)$ defined by (\ref{liniso}), where $\alpha,\beta$  is a pair of invertible elements in the ring $\mathbb{Z}_{m}$ and $d$ is a common divisor of $\mu=\alpha+\beta-1$ and $m$.
\end{theorem}

Classification of linear group isotopes of $\mathbb{Z}_{m}$ according to their symmetry groups is given in the following corollary.

\begin{corollary}\label{cisp}
Let $\mathbb{Z}_{m}$ be a ring of integers modulo $m$ and let $(\alpha,\beta,d)$ be its arbitrary linear isotope, where $d\in GCD(m;\alpha+\beta-1)$. Then an arbitrary linear isotope of an $m$-order cyclic group is isomorphic to exactly one isotope $(\mathbb{Z}_{m},\circ)$ defined by (\ref{liniso}), besides

\begin{center}
\begin{tabular}{|c|c|c|l|}\hline
\myd{a group isotope $(Q;\circ)$} &  \mad{ its symmetry group}  &\mep{\smallskip conditions of its canonical decomposition (\ref{cdec})\smallskip}   \\ \hline
\myd{is strictly commutative}& $\{1,s\}$ & $\beta=\alpha\neq -1$\\ \hline
\myd{is strictly left symmetric} & $\{1,r\}$ & $\beta=-1\neq\alpha$\\ \hline
\myd{is strictly right symmetric}&  $\{1,\ell\}$ & $\alpha=-1\neq\beta$  \\ \hline
\myd{is strictly semi-symmetric} & $A_{3}$ & \mep{\smallskip$\alpha\neq -1$, $\beta=\alpha^{-1}$, \\$\alpha^{3}=-1$, $\alpha d=-d$\smallskip} \\ \hline
\myd{is totally symmetric} & $S_{3}$  & \mep{\smallskip$\alpha=\beta=-1$\smallskip} \\ \hline
\myd{is asymmetric}&  $\{1\}$ & \mep{\smallskip$-1\neq\alpha\neq \beta\neq -1$ and $\beta\neq\alpha^{-1}$ or $\alpha^{3}\neq -1$ or $\alpha{d}\neq -d$\smallskip}  \\ \hline
\end{tabular}
\end{center}
\end{corollary}

\subsection{Linear group isotopes of prime orders}

Group isotopes and linear isotopes were studied by many authors: V.~Belousov~\cite{bil1966}, E.~Falconer~\cite{falc1971},  T.~Kepka and P.~Nemec~\cite{kepnem1971I},~\cite{kepnem1971II} , V.~Shcherbacov~\cite{shcherb1991}, F.~Sokhatsky~\cite{sokha1999}, A.~Dr\'{a}pal~\cite{drapal2009}, G.~Belyavskaya~\cite{belyav2013} and others.

In this  part of the article, we are giving a full classification of linear group isotopes of prime order up to isomorphism relation and according to their symmetry groups.

Theorem~\ref{tliniso2} implies the following statements.

\begin{corollary}\label{cliniso1}$\cite{sokhasyvak1994}$
 Linear group isotopes of a prime order, which are defined by a pair $(\alpha,\beta)$ are pairwise isomorphic, if $\alpha+\beta\ne1$. If $\alpha+\beta=1$, then they are isomorphic to either $(\alpha,\beta,0)$ or $(\alpha,\beta,1)$.
\end{corollary}

\begin{corollary}\label{cliniso2}$\cite{sokhasyvak1994,shchukin2004}$
 There exist exactly $p^{2}-p-1$ linear group isotopes of a prime order $p$ up to isomorphism.
\end{corollary}

Let $p$ be prime, then $\mathbb{Z}_{p}$ is a field, so, according to Corollary~\ref{cliniso1}, there are two kinds of group isotopes of the cyclic group $\mathbb{Z}_{p}$:
 \begin{itemize}
   \item $M_{0}:=\{(\alpha,\beta,0)\mid \alpha,\beta=1,2,\dots,p-1\}$;
   \item $M_{1}:=\{(\alpha,1-\alpha,1)\mid \alpha=2,\dots,p-1\}$.
 \end{itemize}
 For brevity, the symbols $cs$, $ls$, $rs$, $ts$, $ss$, $as$ denote respectively strictly commutative, strictly left symmetric, strictly right symmetric, strictly semi-symmetric, totally symmetric, asymmetric quasigroups. For example, $M_{0}^{ss}$ denotes the set of all strictly semi-symmetric group isotopes from $M_{0}$.

\paragraph{Quasigroups of orders 2 and 3.} Every quasigroup of the orders 2 and 3 is isotopic to the cyclic groups. Only $\iota$ is a unitary substitutions of $\mathbb{Z}_{2}$, so Theorem~\ref{tcdec} implies that there exist two group isotopes:
$$
x\underset{0}{\circ}y:=x+y,\qquad\hbox{and}\qquad x\underset{1}{\circ}y:=x+y+1.
$$
They are isomorphic and $\varphi(x):=x+1$ is the corresponding isomorphism.

\begin{proposition}\label{pisp1}
All quasigroups of the order $2$ are pairwise isomorphic.
\end{proposition}

There exist two unitary substitutions of the group $\mathbb{Z}_{3}$: $\iota$ and $(12)$ and the both of them are automorphisms of the cyclic group $\mathbb{Z}_{3}$. Thus, Theorem~\ref{tcdec} implies that all quasigroups of the order $3$ are linear isotopes of the cyclic group. So,
Corollary~\ref{cliniso1} implies that
$$
M_{0}=\{(1,1,0),(2,2,0),(1,2,0),(2,1,0)\},\qquad M_{1}=\{(2,2,1)\}.
$$
According to Corollary~\ref{cliniso2} and Corollary~\ref{cisp}, we obtain the following result.

\begin{proposition}\label{cisp1}
There exist exactly five 3-order quasigroups up to isomorphism, which can be distributed into four blocks:
\begin{enumerate}
 \item strictly commutative: $(1,1,0)$;
 \item strictly left symmetric: $(1,2,0)$;
\item strictly right symmetric: $(2,1,0)$;
 \item totally symmetric:  $(2,2,0)$, $(2,2,1)$.
\end{enumerate}
\end{proposition}

\paragraph{Linear group isotopes of the order $p>3$.} Full description of these group isotopes is given in the following theorem.

\begin{theorem}\label{tispm} The set of all pairwise non-isomorphic group isotopes of prime order $p\,(p>3)$ is equal to
$$
M=\{(\alpha,\beta,0)\mid \alpha,\beta=1,2,\dots,p\!-\!1\}\cup\{(\alpha,1\!-\!\alpha,1)\mid \alpha=2,\dots,p\!-\!1\}.
$$
The set $M$ equals union of the following disjoint sets:
\begin{enumerate}
 \item the set of all strictly commutative group isotopes
 $$
 M^{cs}=\{(1,1,0),\,(2,2,0),\ldots,(p-2,p-2,0),\,(2^{-1},2^{-1},1)\};
 $$
 \item the set of all strictly left symmetric group isotopes
 $$
 M^{ls}=\{(1,p-1,0),\,(2,p-1,0),\ldots,(p-2,p-1,0),\,(2,p-1,1)\};
 $$
\item the set of all strictly right symmetric group isotopes
$$
M^{rs}=\{(p-1,1,0),\,(p-1,2,0),\ldots,(p-1,p-2,0),\,(p-1,2,1)\};
$$
 \item the set of all totally symmetric group isotopes
 $$
 M^{ts}=\{(p-1,p-1,0)\};
 $$
 \item \label{nn1} the set $M^{ss}$ of all strictly semi-symmetric group isotopes is empty, if $p-3$ is not quadratic residue modulo $p$, but if there exists $k$ such that $p-3=k^{2}$ modulo $p$, then the set is equal to
 $$
 M^{ss}=\left\{\left((1+k)2^{-1},2(1+k)^{-1},0\right),\left((1-k)2^{-1},2(1-k)^{-1},0\right)\right\};
 $$
 \item \label{nn2} the set of all asymmetric group isotopes is equal to
 $$
  M^{as}=\left\{\{(3,p\!-\!2,1),\ldots,(p\!-\!2,3\!-\!p,1)\}\setminus(2^{-1},1\!-\!2^{-1},1)\right\}\!\cup
$$
$$
\cup\left\{(\alpha,\beta,0)\;\big|\; \alpha,\beta=1,2,3,\dots,p-2,\;\alpha\neq\beta\right\}\setminus{M^{ss}}.
$$
\end{enumerate}
\end{theorem}

\begin{proof} Let $(\alpha,\beta,d)$ be an arbitrary group isotope.

If it is commutative, then, according to Corollary~\ref{cisp}, we obtain
$$
M_{0}^{cs}=\{(\alpha,\alpha,0)\mid\alpha=1,2,\ldots,p-2\}, \quad |M_{0}^{cs}|=p-2.
$$
If $d=1$, then Corollary~\ref{cliniso1} implies $2\alpha=1$, i.e., $M_{1}^{cs}=\{(2^{-1},2^{-1},1)\}$. Thus, the set of all pairwise non-isomorphic commutative group isotopes is $M^{cs}=M_{0}^{cs}\cup M_{1}^{cs}$ and $|M^{cs}|=p-1$.

Consider left symmetric quasigroups. According to Corollary~\ref{cisp}, we have
$$
M_{0}^{ls}=\{(\alpha,p-1,0)\mid\alpha=1,2,\ldots,p-2\},\quad |M_{0}^{ls}|=p-2.
$$
If $d=1$, then Corollary~\ref{cliniso1} implies $\alpha=2$ and, by Corollary~\ref{cisp}, $M_{1}^{ls}=\{(2,-1,1)\}$. Thus, the set of all pairwise non-isomorphic left symmetric group isotopes is $M^{ls}=M_{0}^{ls}\cup M_{1}^{ls}$ and $|M^{ls}|=p-1$.

The relationships for right symmetric quasigroups can be proved in the same way.

In virtue of Corollary~\ref{cisp}, an arbitrary totally symmetric isotope is defined by the pair $(-1,-1)=(p-1,p-1)$ of automorphisms. According to Corollary~\ref{cliniso1}, $d=0,1$. Suppose that $d=1$, then $p-1+p-1=1$, i.e., $2p=3$. Since $p>3$, than this equality is impossible, so, $d=0$. Thus, $M^{ts}=\{(p-1,p-1,0)\}$ and $|M^{ts}|=1$.

Consider semi-symmetric quasigroups.  According to Corollary~\ref{cliniso1},
$$
\alpha\ne-1,\qquad \alpha^{3}=-1,\qquad \alpha d=-d,
$$
where $d=0,1$. But $d\ne1$, since $\alpha\ne-1$, so $d=0$. The equality $\alpha^{3}=-1$ is equivalent to $(\alpha+1)(\alpha^2-\alpha+1)=0$. It is equivalent to $\alpha^2-\alpha+1=0$. It is easy to prove that $\alpha$ exists if and only if $p-3$ is a quadratic residue modulo $p$. If $p-3=k^{2}$ modulo $p$, then
$$
M^{ss}=\left\{\left((1+k)2^{-1},2(1+k)^{-1},0\right),\,\left((1-k)2^{-1},2(1-k)^{-1},0\right)\right\}.
$$
Consequently, $|M^{ss}|=2$, and $M^{ss}=\varnothing$ otherwise.

Let $(\alpha,\beta,d)$ be an arbitrary asymmetric group isotope. Corollary~\ref{cliniso1} implies that $d\in\{0,1\}$. Since the given isotope is neither commutative, nor left symmetric, nor right symmetric, nor totally symmetric, then, according to Corollary~\ref{cisp}, $\alpha,\beta\not\in\{0,p-1\}$ and $\alpha\neq\beta$.

In the case when $d=1$, then from Corollary~\ref{cliniso1}, it follows that $\beta=1-\alpha$ and the relationships $\alpha,\beta\not\in\{0,p-1\}$, $\alpha\neq\beta$ imply
$$
\alpha\not\in\{0,1,2,\frac{p+1}{2},p-1\}.
 $$
 Since for $d=1$ semi-symmetric group isotopes do not exist, then
$$
M_{1}^{as}=\left\{(\alpha,1-\alpha,1)\;\Big|\; \alpha=3,\dots,p-2,\;\alpha\neq\frac{p+1}{2}\right\},
$$
$$
|M_{1}^{as}|=p-5.
$$
The set $M_{0}^{as}$ depends on existence of semi-symmetric isotopes. Nevertheless,
$$
M_{0}^{as}=\left\{(\alpha,\beta,0)\;\big|\; \alpha,\beta=1,2,3,\dots,p-2,\;\alpha\neq\beta\right\}\setminus M^{ss},
$$
$$
|M_{0}^{as}|=\left\{\begin{array}{ll}
  p^2-5p+4,& \text{if $p-3$ is a quadratic residue modulo $p$};\\
  p^2-5p+6,& otherwise.
\end{array}\right.
$$
\end{proof}

Note, that F.~Rad\'{o}~\cite{rado1974} proved that a semi-symmetric group isotope of prime order $p$ exists if and only if $p-3$ is a quadratic residue modulo $p$.

 The general formula of all linear pairwise non-isomorphic group isotopes of the prime order $p$ is found in Corollary~\ref{cliniso2} by F.~Sokhatsky and P.~Syvakivskyj~\cite{sokhasyvak1994} and also by K.~Shchukin~\cite{shchukin2004}.

\begin{corollary}\label{cispk} A number of all linear group isotopes of the prime order $p>3$ up to isomorphism is equal to $p^{2}-p-1$ and it is equal to the sum of the following numbers:
\begin{enumerate}
 \item \label{n1} $p-1$ of strictly commutative quasigroups;
  \item \label{n2} $p-1$ of strictly left symmetric quasigroups;
\item  \label{n3} $p-1$ of strictly right symmetric quasigroups;
\item \label{n4} $1$ of the totally symmetric quasigroup;
 \item \label{n5} $2$ of semi-symmetric quasigroups, if $p-3$ is quadratic residue modulo $p$ and $0$ otherwise;
 \item \label{n6} $(p-2)^{2}-5$ asymmetric quasigroups, if $p-3$ is quadratic residue modulo $p$ and $(p-2)^{2}-3$ otherwise.
\end{enumerate}
\end{corollary}

\begin{proof}  The proof immediately follows from the proof of Theorem~\ref{tispm}.
\end{proof}

{\bf Acknowledgment.} {\it The author is grateful to her scientific supervisor Prof. Fedir Sokhatsky for the design idea and the discussion of this article, to the members of his scientific School for helpful discussions and to the reviewer of English  Vira Obshanska.}

\footnotesize{Department of mathematical analysis and differential equations,  \\
         Faculty of Mathematics and Information Technology \\
         Donetsk National University\\
         Vinnytska oblast, Ukraine 21000\\
e-mail: kraynichuk@ukr.net }

\end{document}